\documentclass[12pt,reqno]{amsart}

\usepackage[utf8]{inputenc}
\usepackage[T1]{fontenc}
\usepackage{amsmath,amssymb,amsfonts,amsthm}
\usepackage{mathtools}
\usepackage{mathrsfs}
\usepackage[mathscr]{euscript}

\usepackage{dsfont}
\usepackage{multicol}
\usepackage{booktabs}
\usepackage{xcolor}
\usepackage{color}
\usepackage{tikz}
\usepackage{url}
\usepackage{soul}
\usepackage{setspace}
\usepackage{graphicx}
\usepackage{lmodern}
\usepackage{enumitem}

\usepackage[top=2.5cm,bottom=2.5cm,left=2.5cm,right=2.5cm]{geometry}

\usepackage[numbers]{natbib}

\definecolor{webgreen}{rgb}{0,.5,0}
\definecolor{webbrown}{rgb}{.6,0,0}

\usepackage[
colorlinks=true,
linkcolor=blue,
citecolor=blue,
urlcolor=blue
]{hyperref}

\usepackage[nameinlink]{cleveref}

\numberwithin{equation}{section}

\newtheorem{theorem}{Theorem}[section]

\newtheorem{proposition}[theorem]{Proposition}
\newtheorem{corollary}[theorem]{Corollary}
\newtheorem{lemma}[theorem]{Lemma}

\theoremstyle{definition}

\newtheorem{definition}[theorem]{Definition}
\newtheorem{example}[theorem]{Example}

\theoremstyle{remark}

\newtheorem{remark}[theorem]{Remark}

\newcommand{\be}{\begin{equation}}
	\newcommand{\ee}{\end{equation}}

\newcommand{\bea}{\begin{eqnarray}}
	\newcommand{\eea}{\end{eqnarray}}

\def\s{\sigma}
\def\P{\mathrm{Per}}
\def\S{\mathfrak{S}}

\def\f{\varphi}

\def\Pn{\mathcal{A}}

\def\C{\mathbb{C}}

\def\Z{\mathbb{Z}}

\DeclareMathOperator{\sgn}{sgn}
\DeclareMathOperator{\per}{per}
\DeclareMathOperator{\tr}{Tr}
\DeclareMathOperator{\rk}{rank}

\title[
Schur multipliers in P\'olya conversion problems
]{\boldmath
	The limits of Schur multipliers in P\'olya conversion problems
	for the $q$-permanent function}
\author{Nour-Eddine Fahssi}

\address{
	Faculty of Sciences and Technology,
	Laboratory MSPASI,\\
	Hassan Second University of Casablanca,
	BP 146, 20650, Mohammedia, Morocco
}

\email{noureddine.fahssi@fstm.ac.ma}



\begin{document}
	
	\begin{abstract}
		This paper studies generalized Pólya conversion problems for the $q$-permanent. We establish a sharp threshold governing the transition from low-dimensional algebraic flexibility to higher-dimensional combinatorial rigidity. For $n \ge 3$ and $q \ne \pm 1$, we prove that the $q$-permanent is not linearly convertible to the determinant or permanent. Conversely, we completely classify the space of Schur multiplier preservers for $n=2$. Focusing on Schur multipliers, we characterize the preserver exponents as a $(2n-2)$-dimensional space of additive matrices. We show that for lower Hessenberg matrices, the general geometric obstruction disappears, yielding an explicit determinantal reduction and an $\mathcal{O}(n^3)$ evaluation algorithm. Furthermore, we classify permutational converter exponents, proving that for $n \ge 4$, admissible symmetries are strictly constrained to the dihedral group. Finally, we resolve a mixed conversion problem, showing the solution space is nonempty only for $n \le 4$, which provides a direct algebraic characterization of the $q$-permanent's zero locus in low dimensions.
	\end{abstract}
	
	\maketitle
	
	{
		\setcounter{tocdepth}{1}
		\tableofcontents
	}
				\section{Introduction}
				Let $\mathbb{F}^{n\times n}$ denote the vector space of $n\times n$ matrices over a field $\mathbb{F}$. The determinant and permanent are fundamental matrix functions, yet they exhibit a well-known computational dichotomy: while the determinant can be evaluated efficiently in polynomial time, computing the permanent is \#P-complete. 
				
				Motivated by this disparity, P\'olya~\cite{polya} asked in 1913 whether there exists a uniform rule for assigning signs $\pm 1$ to the entries of a generic matrix $A$ to obtain a matrix $B$ satisfying $\mathrm{per}(A)=\det(B)$. Szeg\"o~\cite{szego} demonstrated that this procedure fails for $n \ge 3$. Marcus and Minc~\cite{minc} generalized this negative result, proving that no linear transformation $\varphi$ satisfies $\mathrm{per}(X)=\det(\varphi(X))$ for $n \ge 3$. A closely related question is the linear preserver problem, which asks us to characterize all linear maps $\varphi$ such that $F(\varphi(X))=F(X)$ for a given matrix function $F$. Frobenius~\cite{frob} characterized determinant preservers, and Marcus and May~\cite{minc1} provided the analogous characterization for the permanent.
				
				In this work, we investigate the conversion problems for the $q$-permanent. For a matrix $A$ and a parameter $q\in\mathbb{C}^{\times}$, the $q$-permanent is defined by~\cite{yang}
				\begin{equation} 
					\P_q(A) \coloneqq \sum_{\sigma \in \mathfrak{S}_{n}} q^{\ell(\sigma)} \prod_{i=1}^{n} a_{i,\sigma(i)},
					\ee
					where $\mathfrak{S}_n$ is the symmetric group and $\ell(\sigma)$ is the inversion length of $\sigma$.  The $q$-permanent interpolates between the permanent ($q=1$) and the determinant ($q=-1$). Marques de S\'a ~\cite[Theorem 3.3]{markes} characterized the group $\mathcal{G}_{q}$ of $q$-permanent preservers for $q\notin\{-1,1\}$. This group is generated by the reversal permutation $\sigma_{0}(i)=n+1-i$ and diagonal scalings.
					
					We define the $q$-permanent to be linearly convertible to the determinant (resp. permanent) if there exists an injective linear map $\varphi$ such that, for all $X\in\mathbb{C}^{n\times n}$,
					\begin{equation}\label{conv0}
						\P_q(X) = \det(\varphi(X)) \quad (\text{resp.}\ \P_q(X) = \per(\varphi(X))).
					\end{equation}
					This formulation provides a natural $q$-analog of the classical P\'olya problem. Furthermore, we formalize the extended linear conversion problem as follows: characterize all linear maps $\varphi$ and scalars $q^{\prime}\in\mathbb{C}^{*}$ (potentially depending on $\varphi$ and $q$) satisfying the identity
					\begin{equation} \label{gen}
						\P_{q}(X) = \P_{q'}(\varphi(X)) \quad \text{for all } X \in \mathbb{C}^{n \times n}.
					\end{equation}
					This framework unifies the classical preserver problem ($q = q'$) and the converter problem ($q \neq q'$).
					
					Specifically, we investigate whether the $q$-permanent permits conversions to the determinant or permanent via \textit{Schur multipliers} (weighted entry-wise scalings) of the form 
					\begin{equation} \label{dil} 
						\varphi: A \longmapsto \left[ z^{\lambda_{i,j}} a_{i,j}\right] = z^\Lambda \circ A, \quad z \in \C^\times,
					\end{equation}
					where $\circ$ is the Hadamard product. For an $n \times n$ matrix $\Lambda = [\lambda_{i,j}]$, we write $z^\Lambda \coloneqq [z^{\lambda_{i,j}}]_{i,j=1}^{n}$.
					If the matrix $\Lambda$ satisfies the conversion or preservation conditions, it is called a \textit{converter exponent} or \textit{preserver exponent}, respectively.  
					
					\paragraph{Organization and contributions}  Section~\ref{s2} establishes that a general linear conversion is impossible for $n \ge 3$ (Theorem~\ref{thm:impossibility}). For $n=2$, we completely classify the linear converters and show that the solution space decomposes into two distinct families (Theorem~\ref{th0}). Section~\ref{s3} characterizes the set of preserver exponents (Theorem~\ref{thm:preserver_properties}) as a $(2n-2)$-dimensional subspace defined by a discrete Monge condition. Section \ref{s4} shows that for lower Hessenberg matrices, the general obstruction disappears, yielding an explicit determinantal conversion and, consequently, an $\mathcal{O}(n^{3})$ evaluation algorithm (Theorem~\ref{thm:hessenberg_converters}). Section \ref{s5} classifies permutational converter exponents, proving that while solutions exist for all permutations when $n \le 3$, admissible symmetries are strictly constrained to the dihedral group $D_n$ for $n \ge 4$ (Theorem~\ref{thm:permutational}). Finally, Section \ref{s6} addresses a mixed conversion problem expressing the $q$-permanent as a linear combination of the determinant and permanent. We show that the corresponding solution space is nonempty if and only if $n \le 4$, which provides a direct algebraic characterization of the $q$-permanent's zero locus in low dimensions (Theorem~\ref{thm:mixed_conversion}). \textbf{Section}~\ref{s7} concludes the paper with a summary and a discussion of possible directions for future research.
					\medskip
					
					\noindent{\textit{Notations}.} Throughout the paper, we assume $q \neq 0$. For each $\sigma \in \mathfrak{S}_n$, $P_\sigma$ denotes the associated $n \times n$ permutation matrix. Let $J_n$ be the all-ones matrix in $\mathbb{C}^{n \times n}$, and let $E_{i,j}$ denote the standard matrix units. For a matrix $M$ and a permutation $\sigma$, we define the diagonal product and the $\sigma$-trace by
					\begin{equation*}
						\mathrm{wt}_\sigma (M) \coloneqq \prod_{i=1}^n m_{i,\sigma(i)}, \qquad 
						\tr_\sigma (M) \coloneqq \sum_{i=1}^n m_{i,\sigma(i)}.
					\end{equation*} 
					The standard matrix trace corresponds to $\sigma = \mathrm{id}$. We denote by $\s_0$ the reversal permutation defined by $\s_0(i)=n+1-i$.
					
					\section{General Impossibility and the \texorpdfstring{$n=2$}{n=2} Classification} \label{s2}
					Our first result extends the P\'olya–Szeg\"o rigidity phenomenon for the permanent to the $q$-permanent.
					
					\begin{theorem} \label{thm:impossibility}
						For $n \ge 3$ and $q \neq -1$, there is no linear transformation that satisfies \eqref{conv0} for all $A\in \C^ {n \times n}$.
					\end{theorem}
					\noindent \textbf{Proof.} 
					Let $q \neq \pm 1$, and let $F$ denote the determinant or permanent function. Suppose that there exists a linear transformation $\varphi$ such that $\P_q(X) = F(\varphi(X))$ for all $X \in \mathbb{C}^{n \times n}$.
					
					We first demonstrate that such a transformation must be invertible. Assume that $\varphi$ is singular, that is, $\f (A)=0$ for some non-zero matrix $A=\left[a_{ij}\right]$. Consequently, the $q$-permanent must be translation-invariant by $A$:  $\P_q(X+A) = F(\varphi(X+A)) = F(\varphi(X)) = \P_q(X)$, for all $X \in \mathbb{C}^{n \times n}.$ Using the fact that row and column permutations change $F$ only by a sign, $F(\Pi X \Sigma) = \pm F(X)$ for permutation matrices $\Pi$ and $\Sigma$, we may assume without loss of generality that $a_{11} \neq 0$. Let $B$ be the matrix defined by: $$ b_{ij} = \begin{cases} -a_{1j} & \text{if } i=1, \\ \delta_{ij} & \text{if } i > 1. \end{cases} $$ The matrix $B+A$ has a zero first row, which implies $\P_q(B+A) = 0$. However, the matrix $B$ itself is upper triangular with diagonal entries $(-a_{11}, 1, \dots, 1)$. Thus, $\P_q(B) = -a_{11} \neq 0$. This contradicts the translation invariance condition $\P_q(B) = \P_q(B+A)$. Therefore, we can infer that $\varphi$ is invertible.
					
					If $T \in \mathcal{G}_q$, then for any $X$,
					$$ F(\varphi T \varphi^{-1}(X)) = \P_q(T \varphi^{-1}(X)) = \P_q(\varphi^{-1}(X)) = F(X). $$
					Thus, the conjugation map $\Phi : T \mapsto \varphi T \varphi^{-1}$ preserves $F$ and yields a group isomorphism $\mathcal{G}_q \cong \mathcal{G}_{\pm 1}$. Since $\varphi$ is a linear isomorphism, $\Phi$ is an isomorphism of linear algebraic groups over $\mathbb{C}$. As such, $\mathcal{G}_q$ and $\mathcal{G}_{\pm 1}$ must share the same dimension for their identity component $\dim(\mathcal{G}^\circ)$ and the same cardinality for their finite group of connected components $|\mathcal{G} / \mathcal{G}^\circ|$.  
					
					According to the characterization by Marques de S\'a~\cite{markes}, for generic $q \notin \{-1, 1\}$, $\mathcal{G}_q$ is generated by left and right diagonal scalings and the reversal permutation $\sigma_0$. Its identity component $\mathcal{G}_q^\circ$ consists of the mappings $X \mapsto D_1 X D_2$ where $D_1, D_2$ are nonsingular diagonal matrices satisfying $\det(D_1) \det(D_2) = 1$. Thus, the algebraic dimension is $\dim(\mathcal{G}_q) = 2n - 1$. Furthermore, its discrete part is generated exclusively by $\sigma_0$ (since $\sigma_0^2 = \mathrm{id}$), meaning $|\mathcal{G}_q / \mathcal{G}_q^\circ| = 2$. 
					
					We now compare this structure to $\mathcal{G}_{\pm 1}$: 
					\begin{itemize}[leftmargin=20pt]
						\item[$\circ$] If $F = \det$, by Frobenius's theorem, $\mathcal{G}_{-1}$ is generated by maps $X \mapsto U X V$ and $X \mapsto U X^T V$ with $\det(UV)=1$. Choosing $U, V \in \mathbb{C}^{n \times n}$ under the determinant constraint gives $2n^2 - 1$ complex parameters. Moreover, the map is unchanged under $(U, V) \sim (\lambda U, \lambda^{-1} V)$ for any $\lambda \in \mathbb{C}^\times$, removing one more degree of freedom. Thus $\dim(\mathcal{G}_{-1}) = 2n^2 - 2$. For $n \ge 3$, we have $2n^2 - 2 > 2n - 1$, so the groups cannot be isomorphic.
						\item[$\circ$] If $F = \mathrm{per}$, the Marcus–May theorem states that the permanent preserver group $\mathcal{G}_{+1}$ consists of maps $X \mapsto D_1 P X Q D_2$ or $X \mapsto D_1 P X^T Q D_2$, where $P,Q$ are permutation matrices and $D_1,D_2$ are diagonal with $\det(D_1 D_2)=1$. Its identity component $\mathcal{G}_{+1}^\circ$ is generated by diagonal scalings, so $\dim(\mathcal{G}_{+1})=2n-1$. The discrete component group is $(\mathfrak{S}_n \times \mathfrak{S}_n)\rtimes \mathbb{Z}_2$ of size $2(n!)^2$ (arising from row/column permutations and transposition). For generic $q\notin\{-1,1\}$ and $n\ge3$, we have $2(n!)^2>2$, so the component groups, and thus the algebraic groups, are not isomorphic.
					\end{itemize}
					In both cases, the algebraic group isomorphism $\mathcal{G}_q \cong \mathcal{G}_{\pm 1}$ yields a contradiction. Thus, no such linear converter $\varphi$ exists for $n \ge 3$.
					\quad $\square$
					\medskip
					
					Our second result provides the complete classification of these transformations for \mbox{$n=2$}.
					
					\begin{theorem}\label{th0}
						Let $\varphi:\mathbb C^{2\times2}\to\mathbb C^{2\times2}$ be a linear transformation. Then $\varphi$ converts the $q$-permanent to the determinant if and only if its matrix representation
						$M\in\mathbb C^{4\times4}$ belongs to one of the following two families.
						
						\begin{enumerate}
							\item[\emph{(I)}]
							There exist $G=\begin{bsmallmatrix}a&b\\ c&d\end{bsmallmatrix}\in \mathbf{GL}_2(\mathbb C)$ and
							$\alpha,\beta,\gamma\in\mathbb C$ such that
							\[
							M=
							\begin{bmatrix}
								G & B\\
								\alpha G & D
							\end{bmatrix},
							\qquad 
							\text{where} \qquad 
							B=
							\begin{bmatrix}
								-\beta a & \beta b/q\\
								-\beta c & \beta d/q
							\end{bmatrix},
							\qquad
							D=
							\begin{bmatrix}
								\gamma a & -\gamma b/q\\
								\gamma c & -\gamma d/q
							\end{bmatrix},
							\]
							and
							\be \label{acc}
							\alpha\beta+\gamma=-\frac{q}{\det(G)}.
							\ee
							
							\item[\emph{(II)}]
							There exist $G=\begin{bsmallmatrix}a&b\\ c&d\end{bsmallmatrix}\in \mathbf{GL}_2(\mathbb C)$ and
							$\mu\in\mathbb C$ such that
							\[
							M=
							\begin{bmatrix}
								0_{2\times2} & B\\
								G & \mu B
							\end{bmatrix},
							\qquad
							\text{where} \qquad 
							B=\frac1{\det(G)}
							\begin{bmatrix}
								aq & -b\\
								cq & -d
							\end{bmatrix}.
							\]
						\end{enumerate}
					\end{theorem}
					
					\noindent \textbf{Proof.}  Let $M$ be the matrix of $\varphi$ with respect to the standard basis of $\mathbb{C}^{2 \times 2}$. We identify a matrix $X = \begin{bsmallmatrix} x & y \\ z & t \end{bsmallmatrix}$ with the vector $w = [x, y, z, t]^T$. With this identification, the quadratic forms $2 \det(X)$ and $2 \P_q(X)$ are represented by the matrices $J(-1)$ and $J(q) = \text{anti-diag}(1, q, q, 1)$, respectively. The condition $\P_q(X) = \det(\varphi(X))$ is equivalent to the matrix congruence:\begin{equation}\label{congru}
						M^T J(-1) M = J(q).
					\end{equation} Since $\det(J(q)) = q^2 \ne 0$, the matrix $M$ is invertible. Let us write $M$ in $2 \times 2$ blocks as $M = \begin{bsmallmatrix} A & B \\ C & D \end{bsmallmatrix}$. To compute the congruence explicitly, we write 
					$$J(-1) = \begin{bmatrix} 0 & E \\ -E & 0 \end{bmatrix}, \quad J(q) = \begin{bmatrix} 0 & Q_q \\ Q_q^T & 0 \end{bmatrix},$$ where $$E = \begin{bmatrix} 0 & 1 \\ -1 & 0 \end{bmatrix}, \quad Q_q = \begin{bmatrix} 0 & 1 \\ q & 0 \end{bmatrix}$$The congruence condition~\eqref{congru} is equivalent to the block system
					\begin{subequations}\label{group}
						\begin{align}
							A^TEC-C^TEA&=0,\label{1}\\
							D^TEB-B^TED&=0,\label{2}\\
							A^TED-C^TEB&=Q_q.\label{3}
						\end{align}
					\end{subequations} We analyze the system by considering the rank of the leading block $A$. In our proof, we will repeatedly use the fundamental $2 \times 2$ matrix identity $X^T E X = \det(X) E$. We analyze the system by considering the rank of the leading block $A$. 
					
					\medskip
					\noindent
					\textbf{Case 1: $A$ invertible.}\; Let $A = G \in \mathbf{GL}_2(\mathbb{C})$. Equation \eqref{1} states that $G^T E C$ is skew-symmetric, hence $G^T E C = \lambda E$ for some $\lambda \in \mathbb{C}$. Thus, we get $E C = \lambda (G^T)^{-1} E$. Since $E^{-1} = -E$, we can use the standard $2 \times 2$ adjugate identity $G^{-1} = -E G^T E / \det(G)$ to obtain: $$C = -\lambda E (G^T)^{-1} E = \alpha G$$ where $\alpha = \lambda / \det(G)$. Substituting $C = \alpha G$ into \eqref{3} gives $G^T E D - \alpha G^T E B = Q_q$, which leads to: $$D = \alpha B - E (G^T)^{-1} Q_q.$$ Now, equation \eqref{2} implies that the matrix $D^T E B$ is skew-symmetric. Substituting our expression for $D$, we get: $$D^T E B = \alpha B^T E B - (E (G^T)^{-1} Q_q)^T E B.$$ Since $B^T E B = \det(B) E$ is skew-symmetric, the remaining term $-Q_q^T G^{-1} B$ must also be skew-symmetric. Therefore, $-Q_q^T G^{-1} B = \beta E$ for some $\beta \in \mathbb{C}$, which yields: $$B = -\beta G (Q_q^T)^{-1} E = \beta G Q_{1/q} E.$$ This explicitly determines $B$ and $D$, yielding the first family of solutions. The scalar $\gamma$ appearing in the expression of $D$ is determined by
					$\alpha$, $\beta$, and $G$, and the relation~\eqref{acc} follows from equation~\eqref{3}.
					
					\medskip
					\noindent
					\textbf{Case 2: $A=\mathbf{0}$.}\; Equation \eqref{1} becomes trivial, while \eqref{3} reduces to $-C^T E B = Q_q$. Since $Q_q$ is invertible, both $B$ and $C$ must be invertible. Writing $C = G \in \mathbf{GL}_2(\mathbb{C})$, we obtain $B = E (G^T)^{-1} Q_q$. Equation \eqref{2} can be rewritten as $(D B^{-1})^T E = E (D B^{-1})$. For $2 \times 2$ matrices, the commutation relation $X^T E = E X$ forces $X$ to be a scalar multiple of the identity. Therefore, $D B^{-1} = \mu I$, giving $D = \mu B$. This yields the second family.
					
					\medskip
					\noindent
					\textbf{Case 3: $\rk (A)=1$.}\; Equation \eqref{1} implies $A^T E C$ is skew-symmetric. Since $\rk(A)=1$, the matrix $A^T E C$  has rank at most 1, which forces $A^T E C = 0$. Consequently, $\rk(C) \le 1$. Since $Q_q$ has rank 2, and the terms $A^T E D$ and $C^T E B$ in Equation \eqref{3} have rank at most 1, $\rk (A^T E D)= \rk (C^T E B)=1$. The invertibility of the full matrix $M$ ensures $B$ and $D$ are invertible. Following the same logic as Case 2, equation \eqref{2} combined with the invertibility of $B$ forces $D = \kappa B$. Multiplying \eqref{3} on the right by $C$ gives: $$Q_q B^{-1} C = \kappa A^T E C - C^T E C$$ Since $\rk(C) \le 1$, we have $C^T E C = \det(C) E = 0$. Using $A^T E C = 0$, we get $Q_q B^{-1} C = 0$. Because $Q_q$ and $B^{-1}$ are invertible, this implies $C = 0$. If $C = 0$, equation \eqref{3} implies $A^T E D = Q_q$. This is a contradiction because $Q_q$ has rank 2, whereas $A^T E D$ has rank at most 1. Therefore, the rank-one case is impossible. \quad $\square$
					\medskip
					
					The classification exhibits two distinct geometric regimes. The first family corresponds to converters with an invertible leading block and is an affine family over $\mathbf{GL}_2(\mathbb C)$. The second family arises from the degenerate case $A=0$. The intermediate rank-one regime is excluded by a rigidity argument combining skew-symmetry and rank constraints.

					\subsection{Interpretation of Theorem~\ref{th0}} We now explain geometrically why there are so many solutions when $n=2$ using an interpretation of the classification.
					
					Theorem~\ref{thm:impossibility} shows that for $n \ge 3$ the hypersurfaces
					$\{\P_q(X)=0\}$ and $\{\det(X)=0\}$ are not isomorphic.
					The case $n=2$ is special: in this dimension, both functions arise from
					nondegenerate bilinear forms $B_q$ and $B_{-1}$ on $\mathbb{C}^4$.
					In this setting, converting the $q$-permanent into the determinant amounts
					to describing the linear isometries
					\[
					\varphi : (\mathbb{C}^4, B_q) \to (\mathbb{C}^4, B_{-1}).
					\]
					
					Let $\mathcal{C}_q$ be the set of such linear converters. The group $\mathbf{O}(\mathbb{C}^4, B_{-1})$ acts on $\mathcal{C}_q$ by right multiplication, and this action is transitive. Consequently, $\mathcal{C}_q$ is a complex smooth algebraic variety of dimension six. From a projective point of view, the associated isotropic cones define quadrics in $\mathbb{P}^3$, and over $\mathbb{C}$ these quadrics are equivalent, which explains the existence of solutions in this case.
					
					Theorem~\ref{th0} gives a decomposition of $\mathcal{C}_q$ into two strata:
					$
					\mathcal{C}_q = \mathcal{C}_q^{(I)} \sqcup \mathcal{C}_q^{(II)},
					$
					corresponding to two orbits distinguished by the rank of the leading block $A$. The first stratum is open and dense; it is six-dimensional and corresponds to an invertible $A$. In this case, the relation~\eqref{acc} ensures the compatibility of the two quadratic structures. The second stratum
					lies on the boundary; it has dimension five and corresponds to $A=0$.
					Geometrically, this can be interpreted as a reflection exchanging the row
					and column spaces. This gives a description of the
					$\mathbf{O}(4,\mathbb{C})$-orbits inside $\mathbf{GL}_4(\mathbb{C})$.
					
					When $n \ge 3$, the sets $\{\P_q(X)=0\}$ and $\{\det(X)=0\}$ are degree-$n$ hypersurfaces rather than quadrics. Their symmetry groups are smaller, and the geometry becomes more rigid. In this regime, no analogous conversion can be expected.
					
					\begin{remark} Theorem~\ref{th0} extends naturally to the general conversion problem for arbitrary parameters $q$ and $q^{\prime}$. By characterizing maps satisfying (1.3), we find that the geometric stratification of the converter space remains unchanged.
					\end{remark}
					
					\section{Structure and Geometry of Preserver Exponents} \label{s3}
					If they exist, any \(q\)-permanent converter is determined up to a preserver. In the following theorem, we describe the set of preserver exponents, that is, the exponents for which the Schur multiplier \eqref{dil} belongs to the group $\mathcal{G}_q$. As this set depends on the modulus of the complex parameter $z$, we introduce the two sets:
					\begin{subequations}
						\begin{align}
							\label{Rn}
							\mathcal{R}_n &= \left\{R \in \mathbb{R}^{n \times n} \colon \forall z \in \mathbb{C} \! \setminus \! \mathbb{S}^1 , \P_q(z^R \circ A) = \P_q(A)\right\}, \\
							\label{Sn}
							\mathcal{S}_n(z) &= \left\{R \in \mathbb{R}^{n \times n} \colon \P_q(z^R \circ A) = \P_q(A)\right\}, \quad \text{for } z \in \mathbb{S}^1.
						\end{align}
					\end{subequations}
					
					\begin{theorem}[Properties of preserver exponents] 
						\label{thm:preserver_properties}
						The space $\mathcal{R}_n$ of preserver exponents has the following properties:
						\begin{enumerate}
							\item $\mathcal{R}_n$ is a vector subspace of the special linear Lie algebra $\mathfrak{sl}_n(\mathbb{R})$.
							\item $\mathcal{R}_n$ is invariant under transposition and under permutations of rows and columns.
							\item Every matrix $R=[r_{i,j}]$ in $\mathcal{R}_n$ can be written in the form $r_{i,j}=u_i+v_j$, where $u=(u_i)$ and $v=(v_j)$ are real vectors satisfying $ \sum_{i} u_i = -\sum_{j} v_j $.
							\item Any $R\in\mathcal{R}_n$ has rank at most $2$. For any nonzero scalar $z$, the matrix $z^R$ has rank $1$.
							\item The dimension of $\mathcal{R}_n$ is $2n-2$.
							\item The subspace $\mathcal{R}_n$ is closed under the symmetric ternary product $(A,B,C) \mapsto ABC + CBA$. In particular, every matrix $R \in \mathcal{R}_n$ satisfies $R^3 \in \mathcal{R}_n$.
							
							\item For a fixed $z=e^{2\pi i\theta}$, where $\theta\in(0,1)$ and $\theta\neq\frac12$, the set $\mathcal{S}_n(z)$ of preserver exponents has the discrete affine decomposition
							\begin{equation}\label{folia}
								\mathcal{S}_n(z) = \bigcup_{\mathbf{k}\in\mathcal{K}} \left( R_0(\mathbf{k},\theta)+\mathcal{R}_n \right),
							\end{equation}
							where $\mathcal{K}\cong\mathbb{Z}^{\,n^2-2n+2}$ is a compatibility lattice and $R_0(\mathbf{k},\theta)$ is a particular solution satisfying $ \tr_\sigma\!\left(R_0(\mathbf{k},\theta)\right) = \frac{k_\sigma}{\theta}$ for all $\sigma\in\mathfrak{S}_n$.
							\item As $z\to 1$ along the unit circle, the Euclidean distance to any non-principal sheet \emph{(}$\mathbf{k}\neq\mathbf{0}$\emph{)} in the decomposition diverges as $\mathcal{O}(1/|\theta|)$. In this limit the compatibility lattice recedes to infinity, and the space reduces to the continuous regime:
							\[
							\lim_{z\to 1}\mathcal{S}_n(z)=\mathcal{R}_n .
							\]
						\end{enumerate}
					\end{theorem}
					\noindent \textbf{Proof.} Let $R = [r_{i,j}] \in \mathcal{R}_n$ be a universal preserver exponent of the $q$-permanent. The relation $z^{\tr_\s(R)} = 1$ must hold for every permutation $\sigma \in \mathfrak{S}n$. If $|z| \neq 1$, then
					\begin{equation} \label{eq:hom_system}
						\tr_\sigma(R) = 0 \qquad \text{for all } \sigma \in \mathfrak{S}_n,
					\end{equation}
					leading to a homogeneous linear system of $n!$ equations in $n^2$ unknowns.
					
					We now state the theorem.
					\begin{enumerate}[label=\arabic*.,leftmargin=20pt]
						\item One can readily check that $\mathcal{R}_n$ is a real vector space. Setting $\sigma = \mathrm{id}$ in \eqref{eq:hom_system} yields $\tr(R) = 0$. Hence, $R \in \mathfrak{sl}_n(\mathbb{R})$.

						\item For any fixed permutations $\tau, \nu \in \mathfrak{S}_n$, we observe that $$\tr_\sigma(P_\tau R) = \tr_{\sigma\tau^{-1}}(R) = 0, \quad   \tr_\sigma(R P_\nu) = \tr_{\nu^{-1}\sigma}(R) = 0, \quad \tr_\sigma(R^\mathsf{T}) = \tr_{\sigma^{-1}}(R) = 0.$$ These relations confirm that the space $\mathcal{R}_n$ is invariant under left and right multiplication by permutation matrices and under transposition.

						\item Let $i_1, i_2$ be distinct row indices and $a, b$ be distinct column indices. By constructing two permutations $\pi_0, \pi_1 \in \mathfrak{S}_n$ that differ only on these indices such that $\pi_0(i_1) = \pi_1(i_2) = a$ and $\pi_0(i_2) = \pi_1(i_1) = b$, the difference of the corresponding trace equations yields:
						\be \label{monge} r_{i_1,a} + r_{i_2,b} = r_{i_1,b} + r_{i_2,a}. \ee
						which is a necessary and sufficient condition for $R$ to admit a decomposition of the form $r_{ij} = u_i + v_j$ for some real vectors $\mathsf{u}=(u_i)_{i=1}^n$, $\mathsf{v}=(v_j)_{j=1}^n$. The condition $\tr(R)=0$ then implies $\sum_{i=1}^n (u_i + v_i) = 0$.

						\item Using the last property, we can write the matrix $R$ in the form $R= \mathsf{u}\mathbf{1}^T + \mathbf{1}\mathsf{v}^T$, where $u, v \in \mathbb{R}^n$ and $\mathbf{1}$ is the all-ones column vector, hence $\rk(R) \le 2$. Furthermore, the matrix $z^R$ has entries $[z^R]_{ij} = z^{u_i} z^{v_j}$. It follows that $z^R$ is a rank-one matrix, being the outer product of the vectors $(z^{u_i})_{i=1}^n$ and $(z^{v_j})_{j=1}^n$. This ensures that $0$ is an eigenvalue of $z^R$ with multiplicity $n-1$.
						\item Let $\psi: \mathbb{R}^n \times \mathbb{R}^n \to \mathbb{R}^{n \times n}$ be defined by $\psi(\mathsf{u}, \mathsf{v}) = [u_i + v_j]$. Obviously, the kernel of $\psi$ is $\operatorname{span}\{(\mathbf{1}, -\mathbf{1})\}$ and therefore $\dim(\operatorname{Im}(\psi)) = 2n-1$. Since $\tr(R)=0$ defines a non-trivial linear functional on this image, we conclude that $\dim \mathcal{R}_n = (2n-1) - 1 = 2n-2$.		
						\item We prove that $\mathcal{R}_n$ is closed under the symmetric ternary product. By a standard polarization argument, it suffices to show that for any $A,B\in\mathcal{R}_n$, $ABA\in \mathcal{R}_n$.
						
						We write $R \in \mathcal{R}_n$ as $R = \mathsf{u} \mathbf{1}^T + \mathbf{1} \mathsf{v}^T$. By applying the translation $u \leftarrow u - c\mathbf{1}$ and $v \leftarrow v + c\mathbf{1}$ with $c = \frac{1}{n} \mathbf{1}^T u$ is the mean of the coefficients $u_i$, we can normalize this representation such that the vectors are trace-free: $\mathbf{1}^T u = \mathbf{1}^T v = 0$.
						
						Let $A= u_a \mathbf{1}^T + \mathbf{1} v_a^T$ and $B= u_b \mathbf{1}^T + \mathbf{1} v_b^T$ two preserver components with normalized representations. Using $\mathbf{1}^T u_b = v_a^T \mathbf{1} = 0$, along with $\mathbf{1}^T \mathbf{1} = n$, we compute $AB$:
						$$ AB = (u_a \mathbf{1}^T + \mathbf{1} v_a^T)(u_b \mathbf{1}^T + \mathbf{1} v_b^T) = n u_a v_b^T + (v_a^T u_b) \mathbf{1} \mathbf{1}^T. $$ Right-multiplying this result by $A$ and applying $\mathbf{1}^T u_a = 0$ and $v_b^T \mathbf{1} = 0$, the cross-terms vanish:
						$$ ABA = n(v_b^T u_a) u_a \mathbf{1}^T + n(v_a^T u_b) \mathbf{1} v_a^T= U \mathbf{1}^T + \mathbf{1} V^T, $$
						where $U= n(v_b^T u_a) u_a$ and $V= n(v_a^T u_b)v_a$. The initial normalizations $\mathbf{1}^T u = 0$ and $\mathbf{1}^T v = 0$ imply $\mathbf{1}^T U + \mathbf{1}^T V = 0$. Hence, by the property 3., $ABA \in \mathcal R_n$. Evaluating this ternary product for $A=B=R$ directly implies $R^3 \in \mathcal{R}_n$. 
						
						\item Let $z=e^{2\pi i\theta}$, where $\theta\in(0,1)$ and $\theta\ne\frac12$. 
						The condition $z^{\tr_\sigma(R)}=1$ is equivalent to
						\[
						\tr_\sigma(R)=\frac{k_\sigma}{\theta},\qquad k_\sigma\in\mathbb Z,
						\]
						which yields a non-homogeneous linear system. Its coefficient matrix 
						$\mathcal A_n\in\{0,1\}^{n!\times n^2}$ is the permutation–entry incidence matrix: the row indexed by $\sigma$ has 1's in the positions $(i,\sigma(i))$. The homogeneous system $\mathcal{A}_n \mathrm{vec}(R) = 0$, is exactly the condition $\tr_\sigma(R) = 0$ for all $\sigma$, which defines the subspace $\mathcal{R}_n$. By the rank-nullity theorem, \footnote{~The algebraic constraints defining $\mathcal{R}_n$ have an unexpected interpretation via competition graphs. Viewing an $n \times n$ matrix as the edge set of the complete bipartite graph $K_{n,n}$, permutations correspond to perfect matchings. The rank of $\mathcal{A}_n$ equals the competition number $k(K_{n,n})$ found by Roberts~\cite{rob}. In the ecological setting from which these graphs arise, $k(G)$ is the minimum number of isolated niches needed to build a valid acyclic food web.} \[
						\rk(\mathcal A_n)=n^2-\dim\mathcal R_n=(n-1)^2+1.
						\] This is the exact number of linearly independent constraints. Let $\mathcal K$ denote the set of integer vectors 
						$\mathbf k=(k_\sigma)_\sigma$ for which the system is consistent. 
						Equivalently,
						$
						\rk(\mathcal A_n\,|\,\mathbf k)=n^2-2n+2.
						$
						For each $\mathbf k\in\mathcal K$ there exists a particular solution 
						$R_0(\mathbf k,\theta)$. Hence the solution set $\mathcal S_n(z)$ is a discrete family of affine spaces parallel to $\mathcal R_n$.

						\item Consider the limit $z\to 1$ (i.e., $\theta\to 0$). Let $R\in\mathcal S_n(z)$ belong to a non-principal sheet. Then there exists $\sigma\in\mathfrak S_n$ with $|k_\sigma|\ge1$. Since $R\mapsto\tr_\sigma(R)$ is a linear functional on $\mathbb R^{n\times n}$ and every linear operator defined on a finite-dimensional normed vector space is bounded, there exists $C>0$ such that $|\tr_\sigma(R)|\le C\|R\|$ for any matrix norm $\|\cdot\|$. Hence,
						\[
						\|R\|\ge \frac{|k_\sigma|}{C|\theta|}.
						\]
						As \(\theta \to 0\), the bound diverges to \(+\infty\). Thus, for each nonzero sheet index \(\mathbf{k}\neq\mathbf0\), the corresponding branch escapes to infinity with growth \(\mathcal O(|\theta|^{-1})\). Hence, in the limit \(\theta\to0\), only the principal sheet \(\mathbf{k}=\mathbf0\) stays bounded, and the solution space reduces to \(\mathcal R_n\).
					\end{enumerate}
					This completes the proof.  \quad $\square$
					
					The subspace $\mathcal{R}_n$ is not closed under the standard Lie bracket or standard matrix multiplication. However, Property 6 of Theorem 3.1 establishes that it is stable under the symmetric ternary product $(A, B) \mapsto ABA$.

					\begin{example}[The preserver space on $\mathbb{S}^1$ for $n=3$] For $n=3$, ordering the permutations as $\s_0 \coloneqq \mathrm{id}$, $\s_1 \coloneqq (23)$, $\s_2 \coloneqq (12)$, $\s_3 \coloneqq (123)$, $\s_4 \coloneqq (132)$, $\s_5 \coloneqq (13)$, the set of preservers on the unit circle is characterized by the non-homogeneous system:
						$\tr_{\s_i}(R) = k_{\s_i}/\theta$ ($k_{\s_i} \in \Z$). For this system to be consistent, it is necessary that $k_{\s_0}+k_{\s_3}+k_{\s_4}=k_{\s_1}+k_{\s_2}+k_{\s_5}$. So, we must satisfy the compatibility condition:
						$$\sum_{\text{$\s$ even}} k_\s = \sum_{\text{$\s$ odd}} k_\s.$$ Thus, the total space of preserver exponents at $z\in \mathbb{S}^1$ is \[ \mathcal{S}_n(z) = \bigcup_{\mathbf{k} \in \mathbb{Z}^5} \left( R_0(\mathbf{k}, \theta) + \mathcal{R}_3 \right),\] where the particular solution is
						{\small \[R_0(\mathbf{k},\theta)=\frac{1}{\theta}\begin{bmatrix}
								0&0&0\\
								0&k_1-k_3&k_2-k_5 \\
								-k_1+ k_3+ k_6 & k_5 & k_3	
							\end{bmatrix}.\]}

						In general, a closed-form particular solution $R_0(\mathbf{k},\theta)$ is impractical for $n \ge 4$, since the vector $\{k_\sigma\}$ must satisfy $n! - n^2 + 2n - 2 \ge 14$ independent linear constraints.  
					\end{example}
					We can exhibit a simple basis of $\mathcal{R}_n$ consisting of the matrices $R_i^{(n)}$ (for $i=1,\dots,n$) and $S_j^{(n)}$ (for $j=2,\dots,n-1$), defined entry-wise by:
					$$
					(R_i^{(n)})_{a,b} = 
					\begin{cases} 
						1, & \text{if } b=i \text{ and } a \ne n, \\ 
						-1, & \text{if } a=n \text{ and } b \ne i, \\ 
						0, & \text{otherwise},
					\end{cases}
					\qquad \text{and} \qquad
					(S_j^{(n)})_{a,b} = 
					\begin{cases} 
						1, & \text{if } a=j, \\ 
						-1, & \text{if } a=n, \\ 
						0, & \text{otherwise}.
					\end{cases}
					$$
					For instance, for $n=3$, this basis is
					$$
					R_1^{(3)} = \begin{bmatrix} 1 & 0 & 0 \\ 1 & 0 & 0 \\ 0 & -1 & -1 \end{bmatrix}, \
					R_2^{(3)} = \begin{bmatrix} 0 & 1 & 0 \\ 0 & 1 & 0 \\ -1 & 0 & -1 \end{bmatrix}, \
					R_3^{(3)} = \begin{bmatrix} 0 & 0 & 1 \\ 0 & 0 & 1 \\ -1 & -1 & 0 \end{bmatrix}, \
					S_2^{(3)} = \begin{bmatrix} 0 & 0 & 0 \\ 1 & 1 & 1 \\ -1 & -1 & -1 \end{bmatrix}.
					$$

					\subsection{Algebraic and combinatorial interpretations}
					\label{sec:geometry}
					
					We now interpret the preserver spaces $\mathcal{R}_n$ and $\mathcal{S}_n(z)$ from Theorem~\ref{thm:preserver_properties}.
					\medskip
					
					\noindent \textbf{The Monge condition.}
					The additive property (3) in Theorem~\ref{thm:preserver_properties} shows that the space $\mathcal{R}_n$ is characterized by a Monge-type condition: for any matrix $R=[r_{i,j}] \in \mathcal{R}_n$, the cross-difference $r_{i,j}+r_{k,l}-r_{i,l}-r_{k,j}$ vanishes for all indices $i, j, k, l$. In combinatorial optimization, this condition ensures that $R$ behaves strictly additively~\cite{burkard1996}.
					
					\medskip
					
					\noindent \textbf{Duality and mass-conserving flows.} In the proof, we showed that the preserver components satisfy the trace identity $\operatorname{Tr}_\sigma(R) = \langle P_\sigma, R \rangle_F = 0$, where $\langle \cdot, \cdot \rangle_F$ is the standard Frobenius inner product. Thus, $\mathcal{R}_n$ is the orthogonal complement of the span of permutation matrices in $\mathbb{R}^{n \times n}$. By Birkhoff's theorem, $\mathcal{R}_n^\perp$ is the set of matrices with uniform row and column sums. Interpreting an $n \times n$ matrix as the edge set of $K_{n,n}$, the relation $r_{i,j} = u_i + v_j$ identifies $\mathcal{R}_n$ with node potentials, while $\mathcal{R}_n^\perp$ corresponds to mass-conserving circulations (cycles) in the associated flow network \cite{ahuja1993}.
					
				\medskip
				
				\noindent \textbf{Algebraic rank condition.} The preserver space $\mathcal{R}_n$ imposes strong structural constraints. While any matrix $R \in \mathcal{R}_n$ has rank at most 2, the entry-wise exponentiated matrix $z^R$ is strictly of rank one. This rank-one property is fundamental to ensuring that the matrix entries factorize as $[z^R]_{ij} = z^{u_i} z^{v_j}$, which is the core mechanism allowing the $q$-permanent to remain invariant under the Schur multiplier transformation.
					
					\section{Exact Conversions for Hessenberg Matrices}\label{s4}
					Although universal conversion is impossible for $n \ge 3$, it is natural to ask whether such a conversion can still hold within structurally restricted classes of matrices. A simple example is triangular matrices, where
					$
					\P_q(A) = \prod_i a_{i,i} = \det(A),
					$
					for any $q$. Similarly, for certain sparse matrices, the $q$-permanent is reduced to a single determinantal term.
					
					A nontrivial example is given by Hessenberg matrices, defined by $a_{i,j}=0$ for $j>i+1$. Their nearly triangular form sharply limits the terms in the $q$-permanent expansion. We show that this limitation yields a positive solution to problem~\eqref{conv0} via a suitable Schur multiplier.
					
					The set of permutations compatible with the Hessenberg pattern is
					\[
					\mathfrak{S}_n^* \coloneqq \{\sigma\in \mathfrak{S}_n : \sigma(i)\le i+1\}.
					\]
					A direct combinatorial argument shows that
					$
					\# \mathfrak{S}_n^*=2^{\,n-1}
					$
					(see~\cite{bru1}) Thus, the number of contributing permutations decreases from $\mathcal{O}(n!)$ growth of the entire symmetric group to $2^{n-1}$. This helps to explain why the Hessenberg structure admits an explicit linear conversion. 
					
					Let $\mathfrak{H}_n$ be the subspace of $n\times n$ lower Hessenberg matrices. Let $q, z \in \mathbb{C} \setminus \{0, 1\}$. We first define the spaces of permanent and determinant converter exponents as
					\begin{align*}
						\mathcal{H}^+_n
						&\coloneqq
						\left\{
						H \in \mathbb{R}^{n \times n} :
						\P_q(A)=\per\bigl(q^{H}\circ A\bigr)
						\ \text{for all } A\in\mathfrak{H}_n
						\right\}, \\
						\mathcal{H}^-_n
						&\coloneqq
						\left\{
						H \in \mathbb{R}^{n \times n} :
						\P_q(A)=\det\bigl((-q)^{H}\circ A\bigr)
						\ \text{for all } A\in\mathfrak{H}_n
						\right\}.
					\end{align*}

					\begin{theorem}[Structure of Hessenberg converter exponents]
						\label{thm:hessenberg_converters}
						Let $q\in\mathbb{C}\setminus\{0,1\}$ and define
						\[
						[H_0]_{ij}=\delta_{j,i+1},
						\qquad
						\mathcal{R}_n^H
						=
						\mathcal{R}_n
						\oplus
						\operatorname{span}\{E_{ij}:j>i+1\}.
						\]
						Then $H_0$ satisfies both the permanent and determinant converter identities on $\mathfrak{H}_n$.
						
						\begin{enumerate}[label=\emph{\arabic*.},itemsep=2pt]
							
							\item If $|q|\neq1$, then
							$
							\mathcal{H}_n^+
							=
							\mathcal{H}_n^-
							=
							H_0+\mathcal{R}_n^H.
							$
							
							\item If $q=e^{2\pi i\theta}$ with $\theta\in(0,1)\setminus\{\frac12\}$, then
							\be \label{H+}
							\mathcal{H}_n^+
							=
							\bigcup_{k\in\mathcal{K}_H}
							\left(
							H_0(k,\theta)+\mathcal{R}_n^H
							\right),
							\ee
							where $\mathcal K_H \subset \Z^{2^{n-1}}$ and
							\be \label{mod}
							\operatorname{tr}_\sigma(H_0(k,\theta))
							=
							\ell(\sigma)+\frac{k_\sigma}{\theta},
							\qquad
							\sigma\in\mathfrak{S}_n^*.
							\ee
							Moreover,
							\begin{equation} \label{eq:H_minus_subset}
								\mathcal{H}_n^- = \left\{ H \in \mathcal{H}_n^+ : \tr_\s(H) \equiv \ell(\s)\!\!\!\!\!\pmod2, \; \forall \sigma \in \mathfrak{S}_n^* \right\}.
							\end{equation}
							
							\item As $q\to1$ along the unit circle,
							\[
							\mathcal{H}_n^+
							\longrightarrow
							H_0+\mathcal{R}_n^H,
							\]
							while every non-principal sheet diverges like $\mathcal{O}(|1-q|^{-1})$.
							
						\end{enumerate}
					\end{theorem}
					\noindent \textbf{Proof.} Let $A \in \mathfrak{H}_n$ be a lower Hessenberg matrix. By definition of $\mathfrak S_n^*$, the expansions of both the permanent and the determinant reduce to permutations $\sigma \in \mathfrak S_n^*$.
					
					We first show that $H_0 \in \mathcal{H}_n^+ \cap \mathcal{H}_n^-$. By definition, $[H_0]_{ij} = \delta_{j,i+1}$. Hence, for any $\sigma \in \mathfrak S_n^*$,
					\[
					\operatorname{Tr}_\sigma(H_0)
					= \sum_{i=1}^n [H_0]_{i,\sigma(i)}
					= \#\{\, i \in \{1,\dots,n-1\} : \sigma(i)=i+1 \,\}.
					\]
					For $\sigma \in \mathfrak S_n^*$, this quantity coincides with $\ell(\sigma)$ (see, for instance, \cite[Chapter~1]{bona}). It follows that
					$
					\P_q(A) = \per\bigl(q^{H_0} \circ A\bigr).
					$
					Moreover,
					\begin{align*}
						\det\bigl((-q)^{H_0} \circ A\bigr)
						&= \sum_{\sigma \in \mathfrak S_n^*} (-1)^{\ell(\sigma)} (-q)^{\operatorname{Tr}_\sigma(H_0)} \, \mathrm{wt}_\sigma(A)= \sum_{\sigma \in \mathfrak S_n^*} q^{\ell(\sigma)} \, \mathrm{wt}_\sigma(A) = \P_q(A),
					\end{align*}
					so that $H_0 \in \mathcal{H}_n^-$ as well.
					
					Let $H \in \mathbb{R}^{n\times n}$. Then $H \in \mathcal{H}_n^+$ if and only if
					$q^{\operatorname{Tr}_\sigma(H)} = q^{\ell(\sigma)},$ for all $\sigma \in \mathfrak S_n^*.$
					
					\medskip
					\noindent
					\textbf{Case $|q|\neq 1$.}
					In this case, the above equality is equivalent to
					$
					\operatorname{Tr}_\sigma(H) = \ell(\sigma),$ for all $\sigma \in \mathfrak S_n^*.$
					Hence $H \in \mathcal{H}_n^+$ if and only if $H \in \mathcal{H}_n^-$, i.e.,
					$
					\mathcal{H}_n^+ = \mathcal{H}_n^-.
					$
					Moreover,
					\[
					\tr_\s(H - H_0) = 0, \qquad \forall \sigma \in \mathfrak S_n^*,
					\]
					so that $H = H_0 \pmod{\mathcal R_n}$. This kernel is precisely
					\[
					\mathcal{R}_n \oplus \mathrm{span}\{E_{ij} : j > i+1\},
					\]
					since entries strictly above the superdiagonal do not contribute to any $\sigma$-trace. This yields the affine structure of $
					\mathcal{H}_n^+ = \mathcal{H}_n^-$.
					
					\medskip
					\noindent
					\textbf{Case $|q|=1$.}
					Write $q = e^{2\pi i \theta}$ with $\theta \in (0,1)$, $\theta \neq \tfrac12$. The condition becomes the modular linear system given in~\eqref{mod}, where $\mathbf{k}=(k_\sigma)_{\sigma \in \mathfrak S_n^*}$.
					
					The compatibility of this linear system is governed by the Rouché--Capelli theorem, and the set of admissible $\mathbf{k}$ forms a lattice $\mathcal{K}_H \subset  \mathbb{Z}^{2^{n-1}}$. For each $\mathbf{k} \in \mathcal{K}_H$, the solution set is a nonempty affine space of the form
					\[
					H_0(\mathbf{k},\theta) + \mathcal{R}_n^H,
					\]
					which proves that $\mathcal{H}_n^+$ is a countable union of affine subspaces parallel to $\mathcal R_n^H$.
					
					Finally, $H \in \mathcal{H}_n^-$ if and only if
					$
					(-1)^{\ell(\s)+\tr_\s(H)} = 1$, for all $\s \in \mathfrak S_n^*$, which is equivalent to the parity condition
					\[
					\tr_\s(H) \equiv \ell(\s) \pmod{2}.
					\]
					
					The proof of the asymptotic limit as $q \to 1$ along the circle is the same as for item 8 in Theorem~\ref{thm:preserver_properties}. This completes the proof. \quad $\square$
					
					\begin{remark} The arithmetic parity condition $k_\sigma \in 2\theta\mathbb{Z}$, given in~\eqref{eq:H_minus_subset} requires attention.
						\begin{itemize}[leftmargin=20pt, itemsep=5pt]
							\item If $\theta=p/r\in\mathbb{Q}$ with $\gcd(p,r)=1$, then $k_\sigma\in 2p\mathbb{Z}$ when $r$ is odd, and $k_\sigma\in p\mathbb{Z}$ when $r$ is even. In this case, $\mathcal{H}_n^-$ is a discrete family of affine sheets separated by multiples of $2p$ or $p$.
							
							\item If $\theta\notin\mathbb{Q}$, then $k_\sigma=0$ for all $\sigma$, and $\mathcal{H}_n^-$ reduces to the principal affine sheet $H_0+\mathcal{R}_n^H$.
					\end{itemize} \end{remark}
					
					Since the $q$-permanent of Hessenberg matrices can be converted to the determinant for all $q$, and the determinant can be computed in $\mathcal{O}(n^3)$ time, the resulting algorithm runs in $\mathcal{O}(n^3)$.
					
					\begin{example}[The Hessenberg converter exponent space on the unit circle for $n=3$] Let \(A \in \mathfrak{H}_3\) and $z= e^{2\pi i \theta}$. The four Hessenberg-compatible permutations are $\mathfrak{S}_3^* = \{\mathrm{id}, (23), (12), (123)\}$. Let $H = [h_{i,j}]$ be a solution of the modular system~\eqref{mod}. This system has eight relevant variables \(h_{i,j}\) (\(h_{1,3}\) is free because \(a_{1,3}=0\)), and its four equations are linearly independent. Thus, five free variables parameterize the solution space, which we take as \(h_{1,1} = a\), \(h_{1,2} = b\), \(h_{2,3} = c\), and \(h_{3,3} = d\) in addition to $h_{1,3}$. Solving the system yields a five-dimensional affine space with elements parametrized as:
						\begin{equation} \label{C+}
							\begin{bmatrix} 
								a & b & h_{1,3} \\[6pt]
								\frac{k_1}{\theta} + 1 - b - d & \frac{k_2}{\theta} - a - d & c \\[6pt]
								\frac{k_3}{\theta} + 2 - b - c & \frac{k_4}{\theta} + 1 - a - c & d 
							\end{bmatrix} \in  H_0(\mathbf{k},\theta) + \mathcal{R}_3^H
							,
						\end{equation} where \(H_0(\mathbf{k},\theta)\) is a particular solution with all free parameters set to zero. The space $\mathcal{H}^+_3$ is given by the union \eqref{H+} where the lattice of admissible integer vectors \(\mathcal{K}_H\) corresponds to \( \mathbb{Z}^4\). Moreover, a matrix $H \in \mathcal{H}^+_3$ belongs to $\mathcal{H}^-_3$ if and only if $k_i/\theta = 2m_i$ for some $m_i \in \mathbb{Z}$ for each $i \in \{1,2,3,4\}$. Substituting $2m_i$ into~\eqref{C+} removes the $\theta$-dependence of the entries. This provides the exact form of any determinant converter exponent. If $\theta$ is irrational, $\mathcal{H}^-_3$  reduces to the main continuous sheet $H_0 + \mathcal{R}_3^H$.  
					\end{example}
					The following proposition gives a dimensional threshold of the compatibility lattice $\mathcal{K}_H$:
					\begin{proposition} \label{prop:lattice_threshold}
						For $n \ge 6$, \mbox{$\mathcal{K}_H \subsetneq \mathbb{Z}^{2^{n-1}}$}. 
					\end{proposition}
					\noindent \textbf{Proof.} The compatibility lattice $\mathcal{K}_H$ is determined by the linear system \eqref{mod}. Define $\mathcal{L} : \mathbb{R}^{n \times n} \to \mathbb{R}^{2^{n-1}}$ by $\mathcal{L}(H) = (\tr_\sigma(H))_{\sigma \in \mathfrak{S}_n^*}$. Since permutations in $\mathfrak{S}_n^*$ never use entries strictly above the superdiagonal, $\mathcal{L}$ factors through the space of lower Hessenberg matrices, of dimension $d_n = n^2 - \binom{n-1}{2}=\frac{1}{2}(n^2+3n-2)$. Hence $\rk(\mathcal{L}) \le d_n$. Although surjectivity is not ruled out for $n=5$ ($d_5 = 19 > 16$), for $n=6$ we have $d_6 = 26 < 32$, and a simple induction shows $d_n < 2^{n-1}$ for all $n \ge 6$. Thus, for $n \ge 6$, $\mathrm{im}(\mathcal{L})$ is a proper subspace of $\mathbb{R}^{2^{n-1}}$, so the target integer vector $\mathbf{k}$ must satisfy additional linear constraints, and therefore $\mathcal{K}_H \subsetneq \mathbb{Z}^{2^{n-1}}$. \quad $\square$
					\medskip
					
					To make these dependencies explicit, we derive a linear relation among the lattice parameters $k_\sigma$ from a linear dependency among permutation matrices in $\mathfrak{S}_n^*$. Consider $\pi_1 = \mathrm{id}$, $\pi_2 = (12)(34)$, $\pi_3 = (12)$, and $\pi_4 = (34)$ in $\mathfrak{S}_4^*$. Their permutation matrices satisfy $P_{\pi_1} + P_{\pi_2} = P_{\pi_3} + P_{\pi_4}$. Hence any exponent matrix $H$ obeys
					$
					\tr_{\pi_1}(H) + \tr_{\pi_2}(H) - \tr_{\pi_3}(H) - \tr_{\pi_4}(H) = 0.
					$
					Using \eqref{mod} and the inversion lengths $\ell(\pi_1)=0$, $\ell(\pi_2)=2$, $\ell(\pi_3)=\ell(\pi_4)=1$, we get
					$
					k_{\mathrm{id}} + k_{(12)(34)} - k_{(12)} - k_{(34)} = 0.
					$
					This shows the $k_\sigma$ are not independent and serves as a low-dimensional prototype of the structural obstructions that increase for $n \ge 6$.
					
					\section{Permutational Symmetries and the Dihedral Threshold}\label{s5}
					The standard determinant and permanent are invariant, up to a sign, under row and column permutations. For the $q$-permanent, this symmetry is deformed: permuting the columns of a matrix requires an entry-wise rescaling combined with a transformation of the parameter $q$.
					
					For instance, using the standard inversion identity $\ell(\s \s_0) = \frac{n(n-1)}{2} - \ell(\s)$, it is straightforward to show the identity 
					\begin{equation} \label{duality} 
						\P_q(AP_{\s_0}) = q^{\frac{n(n-1)}{2}} \P_{q^{-1}}(A) = \P_{q^{-1}}(q^{\frac{n-1}{2}} A), \quad \text{for all} \; A \in \C^{n \times n}.
					\end{equation} This provides a direct solution to \eqref{gen} with $q'=q^{-1}$ and $\varphi(X) = q^{\frac{n-1}{2}} X P_{\sigma_0}$.
					
					\begin{definition}
						Let $\tau \in \mathfrak{S}_n$. The space of \textit{$\tau$-converter exponents} is defined as
						\begin{equation} \label{eq:tau_converter}
							\mathcal{C}^{(n)}(\tau) := \left\{ \Lambda \in \mathbb{R}^{n \times n} : \P_q(A P_\tau) = \P_r(q^\Lambda \circ A) \text{ for all } A \in \mathbb{C}^{n \times n} \right\},
						\end{equation}
						where $r \in \mathbb{C}^{\times}$ is a parameter depending on $q$ and $\tau$.
					\end{definition}
					
					The conversion identity~\eqref{duality} says that the space $\mathcal C^{(n)}(\s_0)$ is not empty for all $n$, and $r(q,\s_0)=q^{-1}$ universally: \[\frac{n-1}{2} J_n \in \mathcal C^{(n)}(\s_0), \quad  \forall n.\]
					The next theorem describes the structure of set~\eqref{eq:tau_converter}.
					\begin{theorem}[Structure of permutational converter exponents]
						\label{thm:permutational}
						Let $\tau \in \mathfrak{S}_n$ and \mbox{$q \notin \{-1,1\}$}.
						
						\begin{enumerate}
							\item If $n \in \{2,3\}$ and $|q|\neq 1$, then
							$
							\mathcal{C}^{(n)}(\tau)$ is an affine space modeled on the preserver component space $\mathcal{R}_n$, with target parameter
							$
							r(q,\tau)=q^{\operatorname{sgn}(\tau)}.
							$
							
							\item Assume $n\ge4$ and $|q|\neq1$. Then
							$$
							\mathcal{C}^{(n)}(\tau)\neq\varnothing
							\quad\Longleftrightarrow\quad
							\tau\in D_n.
							$$
							More precisely, cyclic shifts $\tau=c^k$, $c=(12 \ldots n)$, admit converters with $r(q,\tau)=q$, while reversed cyclic shifts $\tau=c^k\circ\sigma_0$ admit converters with $r(q,\tau)=q^{-1}$. In all cases,
							$
							\mathcal{C}^{(n)}(\tau)
							$ is an affine space modeled on $\mathcal R_n$.
							
							\item If $|q|=1$, let $\omega(q)$ denote the order of $q$.
							\begin{itemize}
								\item If $\omega(q)=\infty$, the classification is identical to the case $|q|\neq1$.
								\item If $\omega(q)=m$, with $3 \le m <\infty$ (i.e., $q$ is an $m$-th root of unity), the trace equations become linear congruences modulo $m$, and $\mathcal{C}^{(n)}(\tau)$ becomes a modular affine lattice.
							\end{itemize}
						\end{enumerate}
					\end{theorem}
					Before proceeding to the proof of Theorem~\ref{thm:permutational}, we establish a combinatorial lemma regarding the standardization of dihedral permutations. 
					
					\begin{lemma}\label{lem:pattern}
						For $n \geq 4$ and $\tau \in \S_n$, the following assertions are equivalent:
						\begin{enumerate}[label=\upshape(\roman*)]
							\item $\tau \in D_n$.
							\item For every subset $I = \{i_1 < i_2 < i_3 < i_4\} \subset \{1,\ldots,n\}$,
							the standardization of $\bigl(\tau(i_1),\tau(i_2),$ $\tau(i_3),\tau(i_4)\bigr)$ belongs to $D_4$.
						\end{enumerate}
					\end{lemma}
					
					\begin{proof}
						$\mathrm{(i)\Rightarrow(ii).}$
						If $\tau = c^k$, then $\tau(i) = ((i-1+k)\bmod n)+1$, so for any \mbox{$I = \{i_1 < \cdots < i_4\}$} the sequence $(\tau(i_1),\ldots,\tau(i_4))$ is an arithmetic progression with difference $1$ modulo $n$. Its standardization is a cyclic shift in $D_4$. The case $\tau = c^k\sigma_0$ is analogous with difference $-1$.
						
						$\mathrm{(ii)\Rightarrow(i).}$
						A permutation belongs to $D_n$ if and only if it strictly preserves cyclic adjacency. Suppose $\tau \notin D_n$. Then there exists an index $k$ such that the cyclic distance between $x = \tau(k)$ and $y = \tau(k+1)$ is strictly greater than 1. Thus, both cyclic intervals between $x$ and $y$ in $\{1, \dots, n\}$ are non-empty. Since $\tau$ is bijective, we can choose $u, v \in \{1, \dots, n\}$ such that $\tau(u)$ and $\tau(v)$ lie in distinct cyclic intervals separating $x$ and $y$. Let $I$ be the naturally ordered set $\{k, k+1, u, v\}$. The elements $k$ and $k+1$ are adjacent in $I$, but their images under $\tau$ separate the images of $u$ and $v$ cyclically. Consequently, the standardization $\tau_I \in \mathfrak{S}_4$ fails to preserve cyclic adjacency, meaning $\tau_I \notin D_4$. This contradiction forces $\tau \in D_n$.
					\end{proof}
					
					\noindent \textbf{Proof of Theorem~\ref{thm:permutational}.} Fix $\tau \in \mathfrak{S}_n$. Writing $r=q^x$, the conversion identity implies
					$
					q^{\operatorname{Tr}_\nu(\Lambda)} = q^{\ell(\tau\circ\nu) - x\ell(\nu)}.
					$ If $|q|\ne 1$ then the above condition is equivalent to the linear system
					\be
					\label{eq:trace_system}
					\operatorname{Tr}_\nu(\Lambda)
					=\ell(\tau\circ \nu) - x\ell(\nu),
					\qquad \forall \nu\in\mathfrak{S}_n.
				\end{equation}
				If $|q|= 1$, let $\omega(q) \in \mathbb{N} \cup \{\infty\}$ be the multiplicative order of $q$, then the condition is equivalent to the linear congruence system:
				\begin{equation}\label{eq:modular_system}
					\operatorname{Tr}_\nu(\Lambda) \equiv \ell(\tau \circ \nu) - x \ell(\nu) \pmod{\omega(q)}, \quad \forall \nu \in \mathfrak{S}_n.
				\end{equation} The question is to determine for which $\tau \in \S_n$ the systems~\eqref{eq:trace_system} and \eqref{eq:modular_system} admit a solution $(\Lambda, x) \in \mathbb{R}^{n \times n} \times \mathbb{R}$.
				\medskip
				
				\noindent 1. \textbf{Case $|q|\ne 1$ and $n \in \{2,3\}$.}
				For $n = 2$, the incidence matrix $\Pn_2$ has rank $2$, and a direct calculation shows that~\eqref{eq:trace_system} is consistent for every $\tau \in \S_2$ with $x = \sgn(\tau)$. Particular solutions are $\Lambda_0^{(2)}(\operatorname{id}) = O_2$ and $\Lambda_0^{(2)}((12)) = \bigl[\begin{smallmatrix}0&0\\1&0\end{smallmatrix}\bigr]$.
				For $n = 3$, the rank of $\Pn_3$ is $5$, and consistency with $x = \sgn(\tau)$ is verified for every $\tau \in \S_3$ by direct substitution, with the base matrices tabulated in Table~\ref{tab:L0}.
				\begin{table}
					\centering
					\renewcommand{\arraystretch}{1.5}
					\setlength{\tabcolsep}{6pt}
					\begin{tabular}{cccccccc}
						$\tau$  && $\mathrm{id}_3 $
						& $(12)$
						& $(23)$
						& $(13)$
						& $(123)$
						& $(132)$ \\
						\hline 
						\noalign{\vskip 5pt}
						$\Lambda_0^{(3)}(\tau)$ && $O_3$
						& {\renewcommand{\arraystretch}{0.9} $\begin{bmatrix}
								0 & 0 & 0 \\
								0 & 0 & -2 \\
								5 & 5 & 1
							\end{bmatrix}$}
						& {\renewcommand{\arraystretch}{0.9}$\begin{bmatrix}
								0 & 0 & 0 \\
								0 & -2 & -2 \\
								7 & 3 & 3
							\end{bmatrix}$}
						& {\renewcommand{\arraystretch}{0.9}$\begin{bmatrix}
								0 & 0 & 0 \\
								0 & 0 & 0 \\
								3 & 3 & 3
							\end{bmatrix}$}
						& {\renewcommand{\arraystretch}{0.9}$\begin{bmatrix}
								0 & 0 & 0 \\
								0 & 2 & 2 \\
								-4 & 0 & 0
							\end{bmatrix}$}
						& {\renewcommand{\arraystretch}{0.9}$\begin{bmatrix}
								0 & 0 & 0 \\
								0 & 0 & 2 \\
								-2 & -2 & 2
							\end{bmatrix}$}\\
						\noalign{\vskip 5pt}
						\hline
					\end{tabular}
					\caption{Base matrices for the affine space $\mathcal C^{(3)} (\tau)$.}
					\label{tab:L0}\end{table}
				\medskip
				
				\noindent 2. \textbf{Case $|q|\ne 1$ and $n \ge 4$.} First, we give a necessary condition of consistency. Let $\pi_1, \pi_2, \pi_3, \pi_4$ be four permutations. We call such a quadruple \emph{balanced} if, for each index $i \in \{1, \ldots, n\}$,
				\begin{equation}\label{eq:bal}
					\{\pi_1(i),\, \pi_4(i)\} = \{\pi_2(i),\, \pi_3(i)\} \quad \text{as multisets.}
				\end{equation}
				For every $\Lambda \in \mathbb{R}^{n \times n}$ and every balanced quadruple,
				\begin{equation}\label{eq:monge}
					\operatorname{Tr}_{\pi_1}(\Lambda) - \operatorname{Tr}_{\pi_2}(\Lambda) - \operatorname{Tr}_{\pi_3}(\Lambda) + \operatorname{Tr}_{\pi_4}(\Lambda) = 0.
				\end{equation}
				Indeed, the left‑hand side is equal to
				$$\sum_{i,j} \lambda_{i,j} \bigl(\delta_{\pi_1(i),j} - \delta_{\pi_2 (i),j} - \delta_{\pi_3(i),j} + \delta_{\pi_4(i),j}\bigr),$$
				which vanishes by condition~\eqref{eq:bal}. This is equivalent to the Monge (rectangle) condition. Consequently, if~\eqref{eq:trace_system} is consistent, the function $$f(\sigma) := \ell(\tau \circ \sigma) - x\,\ell(\sigma)$$ must satisfy~\eqref{eq:monge} for every balanced quadruple.
				\medskip
				
				\noindent \textit{Existence for $\tau \in D_n$.} First, from~\eqref{duality}, we know that $\frac{n-1}{2} J_n$ is a $\sigma_0$-permutational converter with the target parameter $x =-1$. On the other hand, let $c=(12 \dots n)$ be the cyclic shift. By a straightforward counting of inversions, the cyclic shift satisfies~\cite{bjorn}
				\[]
				\ell(c\circ \nu)-\ell(\nu)
				=2\nu^{-1}(n)-n-1,
				\qquad \nu\in\mathfrak{S}_n.
				\]
				The right-hand side is realized by the trace functional associated with the matrix $\Lambda^{(c)}$ defined by
				$
				\Lambda^{(c)}_{i,n}=2i-n-1$, and
				$\Lambda^{(c)}_{i,j}=0$ for $j\neq n.$
				Thus \eqref{eq:trace_system} holds for $\tau=c$ with $x=1$.
				\medskip
				
				\noindent \textit{Closure under composition.}
				Suppose $\tau_1, \tau_2 \in S_n$ admit converter exponents $\Lambda_1, \Lambda_2$ with parameters $x_1, x_2$. Using the identity $\operatorname{Tr}_{\tau_2 \nu}(\Lambda_1) = \operatorname{Tr}_\nu(\Lambda_1 P_{\tau_2}^T)$, we compute for every $\nu \in S_n$:
				$$
				\ell(\tau_1 \tau_2 \nu)
				= x_1\,\ell(\tau_2 \nu) + \operatorname{Tr}_{\tau_2 \nu}(\Lambda_1) = x_1 x_2\,\ell(\nu)
				+ \operatorname{Tr}_\nu\!\bigl(x_1 \Lambda_2 + \Lambda_1 P_{\tau_2}^T\bigr).
				$$
				Hence $\tau_1 \tau_2$ admits the converter exponent $x_1 \Lambda_2 + \Lambda_1 P_{\tau_2}^T$ with parameter $x_1 x_2$.
				Since $D_n = \langle c,\, \sigma_0 \rangle$, an induction on the length of the word in $\{c, \sigma_0\}$ establishes the existence of a converter for every $\tau \in D_n$. The target parameter is $x = 1$ for $\tau = c^k$ and $x = -1$ for $\tau = c^k \sigma_0$.
				\medskip
				
				\noindent \textit{Non‑existence for $\tau \notin D_4$.} Consider $n=4$ and $\tau_0 = (12)$ as a representative case. Suppose $(\Lambda , x)$ satisfies~\eqref{eq:trace_system}. Take the four permutations (in one‑line notation, fixing positions $\geq 5$):
				\[
				\pi_1 = [1,2,4,3], \quad
				\pi_2 = [4,2,1,3], \quad
				\pi_3 = [1,3,4,2], \quad
				\pi_4 = [4,3,1,2].
				\]
				This quadruple is balanced and
				$f(\pi_1)=2-x$, $f(\pi_2)=3-4x$, $f(\pi_3)=3-2x$, $f(\pi_4)=6-5x$, hence $f(\pi_{1})-f(\pi_{2})-f(\pi_{3})+f(\pi_{4})=2$, contradicting~\eqref{eq:monge}. The remaining 15 elements of $\mathfrak{S}_4 \setminus D_4$ split into symmetry classes under left and right multiplication by $D_4$. Since dihedral multiplication preserves balanced quadruples, the geometric obstruction for $\tau_0$ applies to each class. Thus the same contradiction holds for every $\tau \notin D_4$, so $\mathcal{C}^{(4)}(\tau)=\varnothing$.
				
				\paragraph{Reduction to the case $n = 4$ for $n \geq 5$}
				Let $\tau \notin D_n$ for $n \geq 5$. By Lemma~\ref{lem:pattern}, there exists $I = \{i_1 < i_2 < i_3 < i_4\}$ such that the standardization $\tau_I \in \S_4$ of $(\tau(i_1),\ldots,\tau(i_4))$ is not in $D_4$. Set $J = \{j_1 < j_2 < j_3 < j_4\} := \{\tau(i_1),\ldots,\tau(i_4)\}$. We lift the balanced quadruple associated with $\tau_{I}\in\mathfrak{S}_{4}\backslash D_{4}$ to $\mathfrak{S}_{n}$. These lifted permutations fix positions outside $I$ and map $I$ to $J$. This yields a balanced quadruple in $\mathfrak{S}_{n}$ satisfying $f(\pi_{1})-f(\pi_{2})-f(\pi_{3})+f(\pi_{4})=2$, which contradicts~\eqref{eq:monge}. Hence $\mathcal{C}^{(n)}(\tau) = \varnothing$.
				\medskip

				\noindent 3. \textbf{Case $|q|= 1$.} ~
				\medskip
				
				\noindent $-$ \emph{$q$ is not a root of unity.}
				The map $t \mapsto q^t$ is injective on $\mathbb{Z}$. The systems \eqref{eq:trace_system} and \eqref{eq:modular_system} are therefore equivalent, and the above classification applies without change.
				\medskip
				
				\noindent $-$ \emph{$q$ is a primitive $m$‑th root of unity, $m \geq 3$.}
				The system~\eqref{eq:modular_system} becomes a system of congruences modulo $m$. For $\tau \in D_n$, the solutions form a modular affine lattice as stated: \begin{equation*}
					\mathcal{C}^{(n)}(\tau) = \bigcup_{\mathbf{k} \in \mathcal{K}_\tau} \left( \Lambda_0^{(n)}(\tau,\mathbf{k}, m) + \mathcal{R}_n \right)
				\end{equation*}
				For $\tau \notin D_n$ and $n \geq 4$: identity~\eqref{eq:monge} is an integer identity (it does not involve $q$), so the balanced quadruple yields ($\varepsilon_k=\pm 1$)
				\[
				\sum_{k=1}^4 \varepsilon_k\, \operatorname{Tr}_{\pi_k}(\Lambda) = 0
				\qquad \text{while} \qquad
				\sum_{k=1}^4 \varepsilon_k\, f(\pi_k) = 2 \not\equiv 0 \pmod{m}
				\]
				for $m \geq 3$, a contradiction. Thus $\mathcal{C}^{(n)}(\tau) = \varnothing$ for $\tau \notin D_n$ and $m \geq 3$.\quad $\square$
				\medskip
				
				The map $\tau \mapsto x(\tau)$ defines a one-dimensional representation $x:D_n \rightarrow \{-1,1\}$ refining permutation parity: orientation-preserving cyclic shifts lie in the kernel ($x=1$) while reflections yield the alternating character $x=-1$. For $n \ge 4$, this is the unique representation that resolves the non-linear combinatorial constraints imposed by the inversion length. Dihedral permutations uniquely preserve cyclic adjacency. Thus, they alone maintain the consistency of the Schur multiplier system and avoid factorial obstructions in the full symmetric group.
				
				\begin{remark} From Table~\ref{tab:L0}, we see that, for all $\alpha, \beta \in \S_3$,
					\[
					\Lambda_0^{(3)}(\alpha\beta) = \Lambda_0^{(3)}(\beta) + \operatorname{sgn}(\beta)\, \Lambda_0^{(3)}(\alpha).
					\]
					This gives the dualities:
					$
					\Lambda_0^{(3)}(12)+\Lambda_0^{(3)}(132)=\Lambda_0^{(3)}(23)+\Lambda_0^{(3)}(123)=\Lambda_0^{(3)}(13).
					$ 
				\end{remark}
				
				\section{The mixed P\'olya problem}\label{s6}
				After seeing the limits of pure conversion, one may ask whether $q$-permanent can be written as a combination of the determinant and permanent. For $n \in \{3,4\}$ and any matrix $A \in \mathbb{C}^{n \times n}$, one can show that
				\begin{equation} \label{eq:mixed_polya}
					\P_q(A) = \frac{1+q}{2} \, \per(A^+) + \frac{1-q}{2} \, \det(A^-),
				\end{equation}
				where $A^\pm = (\pm q)^M \circ A$ for a specific exponent matrix $M \in \mathbb{R}^{n \times n}$. For example, $M$ may take the values
				\begin{equation} \label{ex4}
					M = \begin{bmatrix}
						0 & 1 & 2 \\
						-1 & 0 & 0\\
						0 & 0 & 0
					\end{bmatrix}
					\quad \text{and} \quad
					M =
					\begin{bmatrix}
						3 & 5 & 6 & 8 \\
						-4 & -2 & -2 & 0 \\
						-2 & -1 & -1 & 0 \\
						0 & 0 & 0 & 0 \\
					\end{bmatrix},
				\end{equation}
				for $n = 3$ and $n = 4$, respectively. This raises the problem of determining whether there exists a universal real exponent matrix $M$ such that Identity~\eqref{eq:mixed_polya} holds for all $A \in \mathbb{C}^{n \times n}$. Interestingly, this mixed P\'olya approach extends previous limits and enables valid conversions up to $n = 4$.
				
				Let $\mathcal{M}_n$ be the set of real exponent matrices that satisfy identity~\eqref{eq:mixed_polya} for all $A \in \mathbb{C}^{n \times n}$.
				\begin{theorem} \label{thm:mixed_conversion}
					Let $q \in \mathbb{C} \setminus \mathbb{S}^1$. The set $\mathcal{M}_n$ is a disjoint union of finitely many affine spaces modeled on $\mathcal{R}_n$. Each component is generated by a base matrix $M_0 \in \Z^{n \times n}$ with \be \label{delta} \delta_\sigma \coloneqq \ell(\sigma) - \tr_\sigma(M_0) \in \{0,1\}, \quad \text{for all $\s \in \S_n$}.\ee The number $a_n$ of such components depends only on $n$, with $a_2 = 4$, $a_3 = 15$, $a_4 = 8$, and $a_n = 0$ for $n \ge 5$, so $\mathcal{M}_n = \varnothing$ in that case.
				\end{theorem}
				\noindent \textbf{Proof.} Let $q\notin \mathbb{S}^1$ and let $M = [m_{ij}] \in \C^{n \times n}$ be an exponent matrix that satisfies the mixed P\'olya problem. Equating the coefficients of both sides of Eq.~\eqref{eq:mixed_polya} yields the following: \begin{equation} \label{eq:coefficient_identity}q^{\ell(\sigma)} = \frac{q^{\tr_\s(M)}}{2} \left[ (1+q) + (1-q) (-1)^{\ell(\sigma) + \tr_\s(M)} \right], \quad \text{for each $\sigma \in \mathfrak{S}_n$}.\end{equation} Assume $\tr_\s(M)\in\mathbb{Z}$ for every $\s$.
				if $\ell(\sigma) + \tr_\s(M)$ is even, \eqref{eq:coefficient_identity} implies that $\tr_\s(M) = \ell(\sigma)$. If $\ell(\sigma) + \tr_\s(M)$ is odd, then $\tr_\s(M) = \ell(\sigma)-1$. Thus, for every $\sigma$, we have the condition~\eqref{delta}. This gives a linear system of the form
				\be 
				\mathcal{A}_n \operatorname{vec}(M) = \mathbf{b}, \label{sys}
				\ee
				where $\mathbf{b} = (b_\sigma) \in \mathbb{Z}^{n!}$ is a target vector with components $b_\sigma \in \{l(\sigma), l(\sigma)-1\}$ and $\operatorname{vec}(M_0)$ is the vector in $\mathbb{R}^{n^2}$ obtained by stacking the entries of $M$ in the same order as used for the rows of $\mathcal{A}_n$. The existence of $M$ is equivalent to the existence of such a $\mathbf{b}$ in the column space of $\mathcal{A}_n$.
				\medskip
				
				\noindent \textit{Small dimensions.} For $n=2$, the exponent matrix 
				$M_0=\begin{bsmallmatrix}0&1\\-1&0\end{bsmallmatrix}$ clearly satisfies Eq.~\eqref{sys} for each of the two permutations in $\mathfrak{S}_2$. There are, in fact, four admissible choices for $M_0$ in this dimension. For $n \in \{3,4\}$, the exponent matrices in~\eqref{ex4} provide explicit solutions. A computer exhaustive search among $2^{3!}$ vectors for $n=3$ shows that 15 choices of $\mathbf{b}$ yield a consistent system, while for $n=4$ we similarly find eight admissible vectors $\mathbf{b}$ among $2^{4!}$ candidates. All such vectors are listed in~\ref{app:target_vectors}. Thus, for $n=2,3,4$, the set $\mathcal{M}_n$ is a disjoint union of 4, 15, and 8 affine spaces modeled on $\mathcal{R}_n$, respectively.
				\medskip
				
				\noindent \textit{Obstruction in $\S_{n \ge 5}$.} The system~\eqref{sys} becomes inconsistent if $n\ge 5$. To prove this, we exhibit a rigid geometric obstruction in $\mathfrak{S}_5$ (which trivially embeds into $\mathfrak{S}_n$ for $n \ge 5$ by fixing indices greater than 5) that conflicts with~\eqref{delta}.  Consider the dihedral subgroup $D_5 \subset \mathfrak{S}_5$. The group $D_5$ partitions into the set $C_5$ of  rotations and the set $R_5$ of reflections. Evaluating the permutation matrices over these two sets, we observe that:
				$$ \sum_{\sigma \in C_5} P_\sigma = \sum_{\sigma \in R_5} P_\sigma = J_5. $$
				Recall that the $\s$-trace of $M$ can be cast in the form: $\tr_\s(M)=\langle P_\s,M\rangle_F$, where $\langle \cdot, \cdot \rangle_F$ is the standard inner matrix product. The linearity of the inner product immediately yields:
				$$ \sum_{\s \in C_5} \tr_\s(M) - \sum_{\s \in R_5} \tr_\s(M) = 0. $$ 
				
				Now, suppose that there exists a universal exponent matrix $M \in \mathcal{M}_n$. Its traces must satisfy $\tr_\s(_M) = \ell(\sigma) -\delta_\s$ for all $\s \in \S_5$. Substituting into the previous identity and simplifying, we get
				$$ \sum_{\s \in C_5} \delta_\s - \sum_{\s \in R_5} \delta_\s = \sum_{\s \in C_5} \ell(\s) - \sum_{\s \in R_5} \ell(\sigma)=20-30=-10. $$ This is an algebraic contradiction. Because $\delta_\sigma \in \{0, 1\}$ for each of the 5 permutations in the respective sets, we have $\sum_{C_5} \delta_\s \ge 0$ and $\sum_{R_5} \delta_\s \le 5$. Therefore, the minimum possible value for the left-hand side is $-5$. The identity forces $-10 \ge -5$, which is impossible. Thus, for $n \ge 5$, no such $\delta_\s \in \{0,1\}$ exists, excluding any universal converter matrix $M \in \mathcal{M}_n$.\quad $\square$

				\begin{remark} When $q  \in \mathbb{S}^1 \setminus \{-1, 1\}$, one can show that each of the $a_n$ affine spaces constituting $\mathcal{M}_n$ splits into a countable lattice of parallel affine sheets. This behavior is strictly analogous to the structure of Hessenberg converter spaces $\mathcal{H}_n^\pm$.  \end{remark}
				
				For $n \le 4$, by differentiating the universal identity~\eqref{eq:mixed_polya} with respect to $q$ and then evaluating at $q = 1$ and $q = -1$, we obtain the unified relation
				\begin{equation} \label{eq:deriv_unified}
					\sum_{\sigma \in \mathfrak{S}_n} \varepsilon^{\ell(\sigma)} \, \delta_\sigma \, \operatorname{wt}_\sigma(A)
					= \frac{1}{2} \Big( \P_\varepsilon(A) - \P_{-\varepsilon}((-1)^M \circ A) \Big),
				\end{equation}
				where $\varepsilon = \pm 1$ and $\delta_\sigma = \ell(\sigma) - \tr_\sigma(M) = \ell(\sigma) - \tr_\sigma(M_0)$ is the Boolean gap between the length of the inversion and the $\sigma$-trace. Equation~\eqref{eq:deriv_unified} connects the determinant ($\varepsilon=-1$) and permanent ($\varepsilon=1$) via the sign matrix $(-1)^M$. 
				\begin{proposition}[Invariants of the sign matrix] \label{prop:invariants_SM}
					For any valid exponent matrix $M \in \mathcal{M}_n$ with $n \le 4$,
					\begin{align*}
						\det(-1)^M &= n! - 2 \sum_{\sigma \in \mathfrak{S}_n} \delta_\sigma, \\
						\operatorname{per}(-1)^M &= -2 \sum_{\sigma \in \mathfrak{S}_n} (-1)^{\ell(\sigma)} \delta_\sigma. 
					\end{align*}
				\end{proposition} 
				\begin{proof} These identities follow from the evaluation of the twin relations given by \eqref{eq:deriv_unified} on the all-ones matrix $A=J$. \end{proof}
				These relations show that the determinant and permanent of the sign matrix $(-1)^M$ are given by a combinatorial count: the number of permutations for which $\tr_\s(M)$ decreases from $\ell(\sigma)$ to $\ell(\sigma)-1$. Table \ref{tab:spectra_SM} shows the resulting spectra of $(-1)^M$.
				\begin{table}
					\centering
					\renewcommand{\arraystretch}{1.3}
					\begin{tabular}{cccc}
						\toprule
						\: $n$ \; & \; $\det(-1)^M$ \; & \; $\operatorname{per}(-1)^M$ \; & \; $\operatorname{Tr}(-1)^M$ \\
						\midrule
						$2$ & $\{-2, 0, 2\}$ & $\{-2, 0, 2\}$ & $\{0, 2\}$ \\
						$3$ & $\{-4, 0, 4\}$ & $\{2\}$ & $\{-1, 1, 3\}$ \\
						$4$ & $\{0\}$ & $\{0\}$ & $\{0, 4\}$ \\
						\bottomrule
					\end{tabular}
					\caption{\small Spectra of the trace, determinant and permanent of the sign matrix $(-1)^M$ evaluated over the solution space $\mathcal{M}_n$.}
					\label{tab:spectra_SM}
				\end{table}
				
				\medskip
				\subsection{The zero locus of the \texorpdfstring{$q$}{q}-permanent in low dimensions}
				\medskip
				
				The zero locus of the $q$-permanent, denoted $\mathcal{V}^{(n)}_q$, is defined as the set of complex matrices $A$ such that \mbox{$\P_q(A)=0$}. The mixed algebraic conversion provides a characterization of the variety $\mathcal{V}_q^{(n)}$ for dimensions $n \le 4$:
				
				\begin{corollary}\label{cor:zero_locus}
					Let $n \le 4$, \mbox{$q \in \mathbb{C} \setminus \{-1, 1\}$}, and $M \in \mathcal{M}_n$. Let $W_q \in \mathbb{C}^{n \times n}$ be the matrix obtained by scaling a single row of $(-1)^M$ by the scalar \mbox{$(q-1)/(q+1)$.} A matrix $A$ belongs to $\mathcal{V}_q^{(n)}$ if and only if its scaled matrix $A^+$ satisfies the generalized P\'olya conversion identity: 
					\be \per(A^+) = \det(W_q \circ A^+).\ee
				\end{corollary} 
				
				\begin{proof} From evaluating the identity \eqref{eq:mixed_polya} at $\P_q(A) = 0$, we see that the zero locus $\mathcal{V}^{(n)}_q$ coincides with the algebraic hypersurface defined by: 
					$$(1+q) \per(A^+) + (1-q) \det(A^-) = 0.$$ 
					The relation $A^- = (-1)^M \circ A^+$ yields $\per(A^+) = \frac{q-1}{q+1} \det((-1)^M \circ A^+)$. Absorbing the scalar $\frac{q-1}{q+1}$ into any single row of $(-1)^M$ defines the matrix $W_q$. This completes the proof.
				\end{proof}
				\begin{example} The mixed conversion identity also yields explicit dense elements of the zero locus \(\mathcal V_q^{(4)}\). For instance, the matrix
					\[
					A=
					\begin{bmatrix}
						1 & -q & -q & -q\\
						1 & 1 & -q & -q\\
						1 & 1 & 1 & -q\\
						1 & 1 & 1 & w(q)
					\end{bmatrix},
					\quad \text{with} \quad 
					w(q)=\frac{q^2(1-q^4-q^5)}{1-q-q^2},
					\]
					satisfies \(\P_q(A)=0\). This provides a nontrivial dense example in the hypersurface \(\mathcal V_q^{(4)}\).
				\end{example}

				
				\section{Concluding Remarks}\label{s7}
				
				In this paper, we established a dimensional threshold for generalized P\'olya conversion problems associated with the $q$-permanent. For $n\le3$, the parameter $q$ admits affine families of Schur-multiplier converters and explicit mixed determinant--permanent representations.
				For $n\ge4$, the constraints induced by inversion length are incompatible with the additive structure of Schur exponents, leading to general nonexistence results. The results suggest several directions for further investigation.
				
				\begin{enumerate}[leftmargin=20pt]
					\item \textit{Mixed conversion identities beyond dimension 4.} The nonexistence of mixed converter spaces for $n \ge 5$ raises the question of whether this threshold is specific to Schur multipliers. It remains an open problem to determine whether other classes of linear transformations might bypass this geometric obstruction.
					
					\item \textit{Classification of sparse matrix patterns admitting exact conversion.} Theorem \ref{thm:hessenberg_converters} gives an explicit determinant conversion for lower Hessenberg matrices, yielding an $\mathcal{O}(n^3)$ polynomial-time algorithm for the $q$ within this class. An open problem is to classify matrix patterns (e.g., Toeplitz-type and interval-pattern matrices) that admit similar polynomial-time linear converters, or to identify the exact threshold where the conversion problem becomes $\#P$-hard.
					
					\item \textit{Connections with algebraic graph theory.}
					The use of permutation–entry incidence matrices and compatibility lattices suggests deeper links with combinatorial optimization and graph theory (cf. Section~\ref{sec:geometry}). In particular, the relation between the incidence matrix rank and competition numbers merits further study.
					
					\item \textit{Extensions to other immanants.} The weight $\sigma \mapsto q^{\ell(\sigma)}$ is not a class function for generic $q$, and does not define an immanant in the classical sense. It is natural to ask if these rigidity phenomena—specifically the dihedral restriction and the $n\ge5$ threshold—extend to standard immanants. Such extensions would involve $\mathrm{Imm}_{\lambda}(A)=\sum_{\sigma\in S_{n}}\chi^{\lambda}(\sigma)\prod_{i}a_{i,\sigma(i)}$ for higher-dimensional irreducible characters $\chi^{\lambda}$ of $\mathfrak{S}_{n}$. One may expect similar geometric obstructions governed by representation-theoretic data.
					
					\item \textit{The zero locus of the $q$-permanent.} 
					Corollary~\ref{cor:zero_locus} characterizes the zero locus $\mathcal{V}_q^{(n)}$ for $n \leq 4$ via a generalized P\'olya conversion identity. For larger $n$ and generic $q$, however, this variety remains largely unexplored. Besides the trivial zeros from the Frobenius--K\"onig condition, the $q$-permanent can also vanish through exact polynomial cancellations. Identifying a structural or graph-theoretic rule for these algebraic zeros remains an open problem.\end{enumerate}
				
				\appendix
				\section{Consistent target vectors for the mixed P\'olya problem} \label{app:target_vectors}
				
				The mixed conversion system \(\mathcal{A}_n \mathrm{vec}(M) = \mathbf{b}_n\) is consistent exactly when the target vector \(\mathbf{b}_n = (b_\sigma)_{\sigma \in \mathfrak{S}_n}\) lies in the column space of \(\mathcal{A}_n\). 
				Among the \(2^{n!}\) binary assignments \(b_\sigma \in \{\ell(\sigma), \ell(\sigma)-1\}\), only those satisfying this condition are valid. For brevity, we omit the base matrices \(M_0\) defining the affine solution spaces; any such matrix can be recovered by finding a particular solution using the target vectors \(\mathbf{b}_3\) and \(\mathbf{b}_4\) below.
				
				For $n=3$, we fix the order of the six permutations by their cycle structure and length as follows: $\sigma_1 = \operatorname{id}$, $\sigma_2 = (23)$, $\sigma_3 = (12)$, $\sigma_4 = (123)$, $\sigma_5 = (132)$, $\sigma_6 = (13)$. Out of $2^6 = 64$ possibilities, $15$ target vectors $\mathbf{b}_3$ yield a consistent system. Their components are given by{\small
					\begin{align*}
						(-1, 0, 0, 1, 2, 2) &  & (0, 0, 1, 2, 2, 3) & & (0, 0, 1, 1, 2, 2) \\
						(-1, 0, 0, 2, 1, 2) &  & (0, 0, 0, 1, 2, 3) & &(0, 1, 0, 2, 2, 3) \\
						(-1, 0, 0, 2, 2, 3) &  & (0, 1, 0, 1, 2, 2) & & (-1, 1, 0, 2, 2, 2) \\
						(-1, 0, 1, 2, 2, 2) &  & (0, 0, 0, 2, 1, 3) &  & (0, 0, 1, 2, 1, 2) \\
						(0, 0, 0, 1, 1, 2)  &  & (0, 1, 0, 2, 1, 2) & & (0, 1, 1, 2, 2, 2) 
					\end{align*}
				} For $n=4$, we take the lexicographic order of the 24 permutations $\sigma_i$: $\operatorname{id}$, $(34)$, $(23)$, $(234)$, $(243)$, $(24)$, $(12)$, $(12)(34)$, $(123)$, $(1234)$, $(1243)$, $(124)$, $(132)$, $(1342)$, $(13)$, $(134)$, $(13)(24)$, $(1324)$, $(1432)$, $(142)$, $(143)$, $(14)$, $(1423)$, $(14)(23)$. Among the $2^{24}$ possible binary choices, the column space of $\mathcal{A}_4$ restricts the solutions to exactly 8 consistent target vectors $\mathbf{b}_4$. They are as follows:
				{\small
					\begin{align*}
						& (0, 1, 1, 2, 2, 2, 0, 1, 2, 3, 3, 3, 1, 2, 2, 3, 4, 4, 3, 3, 4, 4, 5, 5) \\
						& (0, 1, 1, 2, 2, 2, 0, 1, 1, 3, 2, 3, 2, 3, 2, 4, 4, 5, 3, 3, 3, 4, 4, 5) \\
						& (0, 1, 0, 2, 1, 2, 0, 1, 1, 3, 2, 3, 1, 3, 2, 4, 4, 4, 3, 4, 4, 5, 5, 5) \\
						& (0, 1, 0, 1, 2, 2, 0, 1, 1, 3, 3, 4, 1, 2, 2, 4, 4, 5, 3, 3, 4, 5, 4, 5) \\
						& (0, 0, 1, 2, 1, 2, 1, 1, 2, 3, 2, 3, 2, 3, 2, 3, 4, 4, 3, 4, 3, 4, 5, 5) \\
						& (0, 0, 1, 1, 2, 2, 1, 1, 2, 3, 3, 4, 2, 2, 2, 3, 4, 5, 3, 3, 3, 4, 4, 5) \\
						& (0, 0, 0, 1, 1, 2, 1, 1, 2, 3, 3, 4, 1, 2, 2, 3, 4, 4, 3, 4, 4, 5, 5, 5) \\
						& (0, 0, 0, 1, 1, 2, 1, 1, 1, 3, 2, 4, 2, 3, 2, 4, 4, 5, 3, 4, 3, 5, 4, 5)
					\end{align*}
				}

			\end{document}